\documentclass[a4paper,12pt]{amsart}

\usepackage{amsfonts,bm}
\usepackage{amssymb}
\usepackage{setspace}
\usepackage{fancyhdr}
\usepackage{color}
\usepackage{rotating}
\usepackage{mathdots}
\usepackage[tableposition=top]{caption}
\usepackage{picinpar,graphicx}
\usepackage{amscd}
\usepackage{indentfirst}
\usepackage{latexsym}
\usepackage{amsthm}
\usepackage{amsmath}
\newtheorem{theo}{Theorem}[section]
\newtheorem{prop}{Proposition}[section]
\newtheorem{lemma}{Lemma}[section]
\newtheorem{cor}{Corollary}[section]
\newtheorem{remark}{Remark}[section]

\newcommand{\mcp}{\mathcal{P}}

\newcommand{\G}{\mathcal{G}_0}

\DeclareMathOperator{\card}{card}
\DeclareMathOperator{\sh}{Sh}

\DeclareMathOperator{\Cs}{Cs}
\DeclareMathOperator{\Ci}{Ci}
\DeclareMathOperator{\C}{C}

\begin{document}
\title[On the connections of generalized entropies]{On the connections of generalized entropies with Shannon and Kolmogorov-Sinai entropies}
\author{Fryderyk Falniowski}

\address[F.~Falniowski]{Department of Mathematics, Cracow University of Economics,
Rakowicka~27, 31-510 Krak\'ow, Poland}
\email{fryderyk.falniowski@uek.krakow.pl}
\date{\today}



\maketitle

\begin{abstract}
We consider the concept of generalized measure-theoretic entropy, where instead of the Shannon entropy function we consider an arbitrary concave function defined on the unit interval, vanishing in the origin. Under mild assumptions on this function we show that this isomorphism invariant is linearly dependent on the Kolmogorov-Sinai entropy.
\end{abstract}

\section{Introduction}
Dynamical and measure-theoretic (called also Kolmogorov-Sinai entropy) entropies are basic tools for investigating dynamical systems (see e.g. \cite{Downar,KH}). They were  extensively studied  and successfully applied among others in statistical physics  and quantum information. It appeared to be an exceptionally powerful tool for exploring nonlinear systems.
One of the biggest advantages of the Kolmogorov-Sinai entropies lies in the fact that it makes possible to distinguish the formally regular systems (those with the measure-theoretic entropy equal to zero) from the chaotic ones (with positive entropy, which implies positivity of topological entropy \cite{Misiurewicz}).

 The Kolmogorov-Sinai entropy of a~given transformation $T$ acting on a~probability space $(X,\Sigma,\mu)$ is defined  as the supremum over all finite measurable partitions $\mathcal{P}$   of the dynamical entropy of $T$ with respect to $\mathcal{P}$, denoted by $h(T,\mathcal{P})$.
As a dynamical counterpart of Shannon entropy, the entropy of transformation $T$ with respect to a~given partition $\mathcal{P}$ is defined  as the limit of the sequence $\left(\frac 1n H(\mathcal{P}_n)\right)_{n=1}^{\infty}$, where
\[
H(\mathcal{P}_n)=\sum\limits_{A\in \mathcal{P}_n} \eta\left(\mu(A)\right)
\]
with $\eta$ being the Shannon function given by $\eta(x)=-x\log x$ for $x>0$ with $\eta(0):=0$
and $\mathcal{P}_n$ is the join partition of partitions $T^{-i}\mathcal{P}$ for $i=0,...,n-1$.
The existence of the limit in the definition of the dynamical entropy follows from the subadditivity of $\eta$. The most common interpretation of this quantity is the average (over time and the phase space) one-step gain of information about the initial state. Taking supremum over all finite partitions we obtain an isomorphism invariant  which measures  the rate of producing randomness (chaos) by the system. 

Since Shannon's seminal paper \cite{Shannon}   many  generalizations of the concept of Shannon static entropy were considered, see Arimoto \cite{Arimoto}, R\'{e}nyi \cite{Renyi} and Csisz\'ar's survey article \cite{Csiszar2008}.
The dynamical and measure-theoretic counterparts were considered by few authors. De Paly \cite{de Paly1} proposed generalized dynamical entropies based on the concept of the relative static entropies. Unfortunately it appeared that, despite some special cases \cite{de Paly1, de Paly2} the explicit calculations of this invariant may not be possible. 
 Grassberger and Procaccia  proposed in \cite{GP83} a~dynamical counetrpart of the well-known generalization of Shannon entropy -- the R\'{e}nyi entropy, and its measure-theoretic counterpart were considered by Takens and Verbitski. They showed  that for ergodic transformations with positive measure-theoretic entropy, R\'{e}nyi entropies of a~measure-theoretic transformation  are either infinite or equal to the measure-theoretic entropy \cite{TakensVerb1998}. The answer for non-ergodic aperiodic transformations is different, for R\'{e}nyi entropies of order $\alpha >1$ they are equal to the essential infimum of the measure-theoretic entropies of measures forming the decomposition of a~given measure into ergodic components, while for $\alpha<1$ they are still  infinite \cite{TV2002}. In particular, this means that R\'{e}nyi entropies of order $\alpha<1$ are metric invariants sensitive to ergodicity.
Similar generalization was made by  Mes\'{o}n and Vericat \cite{MV96, MV} for so called Havrda-Charv\'{a}t-Tsallis entropy \cite{Havrda} and their results were similar to ones obtained by Takens and Verbitski in \cite{TakensVerb1998}.

Our approach is based on  Arimoto generalization applied to dynamical case. Instead of the Shannon function $\eta$ we consider a~concave function $g\colon [0,1]\mapsto \mathbb{R}$ such that $\lim\limits_{x\to 0^+}g(x)=g(0)=0$ and define the dynamical $g$-entropy of the finite partition $\mcp$ as \[h(g,T,\mcp)=\limsup_{n\to\infty}\frac 1n \sum_{A\in\mcp_n}g(\mu(A)).\]

The behaviour of the quotient $g(x)/\eta(x)$ as $x$ converges to zero appears to be crucial for our considerations. Mainly, defining \[\Ci (g):=\liminf_{x\to 0^+}\frac{g(x)}{\eta(x)}\;\;\text{and}\;\;\Cs(g):=\limsup_{x\to 0^+}\frac{g(x)}{\eta(x)}\] we will prove that \[\Ci(g)\cdot h(T,\mcp)\leq h(g,T,\mcp)\leq \Cs(g) \cdot h(T,\mcp).\] 
In the case of $\Ci(g)=\infty$ we will show that in every aperiodic system and for every $\gamma \geq 0$, there exists a finite partition $\mcp$ such that $h(g,T,\mcp)\geq \gamma$.

Taking the supremum over all partitions we obtain Kolmogorov entropy-like isomorphism invariant, which we will call the measure-theoretic $g$-entropy  of a~transformation with respect to an invariant measure. One might ask whether this invariant may give any new information about the system. We will prove (Theorem \ref{twhKSg}) that for  $g$  with $\Cs(g)<\infty$, this new invariant is linearly dependent on Kolmogorov-Sinai entropy. It means that in fact the Shannon entropy function is the most natural one -- not only it has all of the properties which the entropy function should have \cite{Downar}, but also considering different entropy functions we will not obtain essentially different invariant.
This result might has the other interpretation.
Ornstein and Weiss showed in \cite{OrnsteinWeiss} that every finitely observable invariant for the class of all ergodic processes has to be a~continuous function of the entropy. It is easy to see that any continuous function of the entropy is finitely observable -- one simply composes the entropy estimators with the continuous function itself. In other words an isomorphism invariant is finitely observable if and only if it is a~continuous function of the Kolmogorov-Sinai entropy. Therefore our result implies that the generalized measure-theoretic entropy is in fact finitely observable. It should be possible  to give a~more direct proof of the finite observability of the generalized measure-theoretic entropy but  the proof cannot be easier\footnote{Benjamin Weiss personal communication} than the proof that entropy itself is finitely observable, see \cite{Weiss}.

The paper is organized as follows: in the next section we give a~formal definition of the dynamical $g$-entropy and establish its basic properties. The subsequent section is devoted to the construction of a zero dynamical entropy process with a~given positive $g$-entropy.
 Finally, in the last section, we define a~measure-theoretic $g$-entropy of a~transformation and show connections between this new invariant and the Kolmogorov-Sinai entropy.

 \section{Basic facts and definitions}
Let $(X,\Sigma,\mu)$ be a Lebesgue space and let $g:[0,1] \mapsto \mathbb{R}$ be a~concave function with $g(0)=\lim\limits_{x\to 0^+}g(x)=0$.\footnote{We might assume only that $g(0)=0$, but then the idea of the dynamical $g$-entropy would fail, since if $\mcp_{n+1}\neq \mcp_n$ for every $n$ and $\lim\limits_{x\to 0^+}g(x)\neq 0$, then the dynamical $g$-entropy of the partition $\mcp$ would be infinite. Therefore, if $g$ is not well-defined at zero we will assume that $g(0):=\lim\limits_{x\to 0^+}g(x)$.} By $\mathcal{G}_0$ we will denote the set of all such functions. Every $g\in\G$ is subadditive, i. e.  $g(x+y)\leq g(x)+g(y)$ for every $x,y\in[0,1]$, and quasihomogenic, i.e.  $\varphi_g\colon (0,1]\to\mathbb{R}$ defined by $\varphi_g(x):=g(x)/x$ is decreasing (see \cite{Rosenbaum}).\footnote{If $g$ is fixed we will omit the index, writing just $\varphi$.}

For a~given finite partition $\mathcal{P}$ 
 we define the \textsf{$g$-entropy of the partition $\mathcal{P}$} as
\begin{equation}
H(g,\mathcal{P}):=\sum_{A\in\mcp} g\left(\mu(A)\right).
\end{equation}

For $g=\eta$ the latter is equal to the Shannon entropy of the partition $\mathcal{P}$. For two finite partitions
$\mathcal{P}$ and $\mathcal{Q}$ of the space $X$ we define a new partition
$\mathcal{P}\vee\mathcal{Q}$ (\textsf{join partition} of $\mathcal{P}$ and $\mathcal{Q}%
$) consisting of the subsets of the form $B\cap C$ where $B\in\mathcal{P}$ and
$C\in\mathcal{Q}$. The join partition of more than two partitions is defined similarly.

\subsection{Dynamical $g$-entropies.} For an automorphism $T\colon X\mapsto X$ and a partition $\mathcal{P}=\{E_{1},...,E_{k}\}$ we put
\[
T^{-j}\mathcal{P}:=\{T^{-j}E_{1},...,T^{-j}E_{k}\}
\]
and
\[
\mathcal{P}_{n}=\mathcal{P}\vee T^{-1}\mathcal{P}\vee...\vee T^{-n+1}%
\mathcal{P}.%
\]

Now for a~given $g\in\G$ and a finite partition $\mcp$ we can define the \textsf{dynamical $g$-entropy of the transformation $T$ with respect to $\mathcal{P}$}  as

\begin{equation}
\label{dynentr} h_{\mu}(g,T,\mcp)=\limsup_{n\to\infty}\frac 1n H\left(g,\mcp_n\right).
\end{equation}
Alternatively we will call it the $g$-entropy of the process $(X,\Sigma,\mu,T,\mcp)$. If the dynamical system $(X,\Sigma,T,\mu)$ is fixed then we omit $T$,  writing just $h(g,\mcp)$.  As in the case of Shannon dynamical entropies we are interested in the existence of the limit of $\left(\frac 1n H(g,\mcp_n)\right)_{n=1}^{\infty}$. If $g=\eta$, we obtain the Shannon dynamical entropy $h(T,\mcp)$. However, in the general case we can not replace an upper limit  in (\ref{dynentr}) by the limit, since it might not exist. Existence of the limit in the case of the  Shannon function follows from the subadditivity of the static Shannon entropy. This property has every subderivative function, i.e. function for which the inequality $g(xy)\leq xg(y)+yg(x)$ holds for any $x,y\in[0,1]$, but this is not true in general (an appropriate example will be given in Section 2.2). Therefore we propose more general classes of functions for which the limit exists. It  exists if $g$ belongs to one of two following classes: 
\[
\mathcal{G}_{0}^{0}:=\left\{g\in \mathcal{G}_0\;\left|\; \lim_{x\to 0^+} \frac{g(x)}{\eta(x)}=0\right.\right\}
\;\;\;\text{or}\;\;\;
\mathcal{G}_{0}^{\sh}:=\left\{g\in \mathcal{G}_0\;\left|\; 0<\lim_{x\to 0^+} \frac{g(x)}{\eta(x)}<\infty \right.\right\}.
\]
It is easy to show that if $g$ is subderivative then the limit $\lim\limits_{x\to 0^+}g(x)/\eta(x)$ is finite.
  Moreover we will see that values of dynamical $g$-entropies depend on the behaviour of  $g$ in the neighbourhood of zero.
We will prove that if $g\in\G^0\cup\G^{\sh}$, then there is a linear dependence between the dynamical $g$-entropy and the Shannon dynamical entropy of a given partition. 
Before we give the general result (Theorem \ref{hgp}) we will state few facts, which we will use in the proof of this theorem. We give the following lemmas ommiting their elementary proofs.
\begin{lemma}\label{lemmaalgebra}
Let $b_i> 0$, $a_i\in\mathbb{R}$ for $i=1,\ldots,m$. Then \[\min_{i=1,\ldots,m}\frac{a_i}{b_i}\leq \frac{\sum_{i=1}^m a_i}{\sum_{i=1}^m b_i}\leq \max_{i=1,\ldots,m}\frac{a_i}{b_i}.\]
\end{lemma}
\begin{lemma}
If $\mcp\in\mathfrak{B}$, $\delta>0$, and $g\colon [0,1]\mapsto \mathbb{R}$, then \begin{equation} \label{Gdelta}  \sum\limits_{ A\in\mcp,\; \mu(A)\geq \delta}g(\mu(A)) \leq \frac{1}{\delta}\max_{x\in [\delta,1]}g(x).\end{equation}
\end{lemma}
The following lemma states that the value of the dynamical $g$-entropy is given by the behaviour of $g$ in the neighbourhood of zero.

\begin{lemma} \label{g1rowng2}
If $g_1,g_2\in\G$ and there exists $c>0$ such that $g_1(x)=g_2(x)$ for $x\in[0,c]$, then for every $\mcp\in\mathfrak{B}$  $h(g_1,\mcp)=h(g_2,\mcp)$.
\end{lemma}
\begin{proof}
Let $\mcp\in\mathfrak{B}$ and $g_1,g_2\in\G$, $c>0$ be fullfill the assumptions. Since  $g\in\G$ is bounded we have
\begin{eqnarray}
|H(g_1,\mcp_n)-H(g_2,\mcp_n)|&=&\left|\sum_{A\in\mcp_n:\;\mu(A)>c}(g_1(\mu(A))-g_2(\mu(A)))\right| \nonumber \\  &\leq & \frac 1c \max_{x\in[c,1]}|g_1(x)-g_2(x)|. \nonumber
\end{eqnarray} Dividing by $n$ and converging to infinity we obtain \[h(g_1,\mcp)=h(g_2,\mcp).\]
\end{proof}
We may state now the main theorem of this section.
\begin{theo} \label{hgp} Let $\mcp\in\mathfrak{B}$.
\begin{enumerate}
\item If  $g\in\G$ is such that $g'(0)<\infty$, then $h(g,\mcp)=0$.
\item If $g_1,g_2\in\G$ are such that $g_1'(0)=g_2'(0)=\infty$, \[\liminf\limits_{x\to 0^+}\frac{g_1(x)}{g_2(x)}<\infty,\] and $h(g_2,\mcp)<\infty$, then 
\[\liminf_{x\to 0^+}\frac{g_1(x)}{g_2(x)}\cdot h(g_2,\mcp)\leq h(g_1,\mcp).\]
If additionally $\limsup\limits_{x\to 0^+}\frac{g_1(x)}{g_2(x)}<\infty$, then
\[ h(g_1,\mcp)\leq \limsup_{x\to 0^+}\frac{g_1(x)}{g_2(x)} \cdot h(g_2,\mcp).\]
\item  If $h(g_2,\mcp)=\infty$ and $\liminf\limits_{x\to 0^+}\frac{g_1(x)}{g_2(x)}>0$, then $h(g_1,\mcp)=\infty$.
\end{enumerate}
\end{theo}

\begin{remark}
Whenever $g_2\colon [0,1]\mapsto \mathbb{R}$ is a~nonnegative concave function  satisfying $g_2(0)=0$ and $g_2'(0)=\infty$, we can have any pair $0< a\leq b \leq \infty$ as limit inferior and limit superior of $g_1/g_2$ in 0, choosing a suitable function $g_1$. The idea is as follows: construct $g_1$ piecewise linear. To do so define inductively a strictly decreasing sequence $x_k\to 0$, and a decreasing sequence of values $y_k=g_1(x_k)\to 0$, thus defining intervals $J_k:=[x_{k+1},x_k]$ where $g$ is affine. The only constraint to get a~concave function is that the slope of $g$ on each interval $J_k$ has to be smaller than $y_k/x_k$, and increasing with respect to $k$; this is not an obstruction to approach any limit inferior and limit superior for $g_1(x)/g_2(x)$, provided that $x_{k+1}>0$ is choosen small enough.
\end{remark}

\begin{proof}[Proof of Theorem \ref{hgp}] Let $\mcp\in\mathfrak{B}$. Suppose that $g\in\G$ and $g'(0)<\infty$. Then 
\begin{eqnarray}
h(g,\mcp)&=&\limsup_{n\to\infty}\frac 1n H(g,\mcp_n)\leq \limsup_{n\to\infty}\frac 1n \varphi\left(\frac{1}{\card \mcp_n}\right) \leq\lim_{n\to\infty} \frac {g'(0)}{n} =0,\nonumber
\end{eqnarray}
which completes the proof of point 1. To prove point 2 let $g_1,g_2\in\G$ be such that $g_1'(0)=g_2'(0)=\infty$ and $h(g_2,\mcp)<\infty$. W.l.o.g we can assume that $g_1(x),g_2(x)>0$ for $x\in (0,1)$, since if there exists $x_0\in(0,1)$ such that $g_i(x_0)=0$ for $i=1$ or $i=2$, then we can define  $\tilde{g}_i\colon [0,1]\mapsto \mathbb{R}$ as \[\tilde{g}_i(x):=\begin{cases}g_i(x), & \text{for} \; x\in[0,s_i) \\ g_i(s_i), & \text{for}\; x\in[s_i,1] \end{cases}\] 
where $s_i\in (0,1]$ is such that $\max\limits_{x\in[0,1]}g(x)=g(s_i)$. Function $\tilde{g}$ is strictly positive and by  Lemma \ref{g1rowng2} we have \[h(\tilde{g}_i,\mcp)=h(g_i,\mcp).\] 

We will assume that
\[\limsup\limits_{x\to 0^+}\frac{g_1(x)}{g_2(x)}<\infty.\] 
The estimation  of the lower boundary for $h(g_1,\mcp)$  remains correct if we omit this assumption.
Since $g$~is subadditive, the sequence $(H(g,\mcp_n))_{n=1}^{\infty}$ is nondecreasing and there exists the limit $\lim\limits_{n\to \infty} H(g_2,\mcp_n)$. If it is finite, then $h(g_2,\mcp)=0$ and by (\ref{Gdelta}) and Lemma \ref{lemmaalgebra} we have 
 \begin{eqnarray}\sum_{A\in\mcp_n}g_1(\mu(A))&\leq &\sum_{A\in\mcp_n:\;\mu(A)<\frac 12}g_1(\mu(A))+2\max_{x\in[\frac 12,1]}g_1(x)\nonumber \\ &\leq & \sup_{x\in(0,\frac 12)}\frac{g_1(x)}{g_2(x)}\cdot \sum_{A\in\mcp_n:\; \mu(A)<\frac 12}g_2(\mu(A))+ 2\max_{x\in[\frac 12,1]}g_1(x). \nonumber \end{eqnarray} Since $\limsup\limits_{x\to 0^+}\frac{g_1(x)}{g_2(x)}<\infty$, there exists $M>0$ such that $g_1(x)/g_2(x)<M$ for $x<1/2$. Therefore $\sup\limits_{x\in(0,\frac 12)}\frac{g_1(x)}{g_2(x)}<\infty$, and by Lemma  \ref{lemmaalgebra} we obtain \begin{eqnarray} 0\leq h(g_1,\mcp)&= &\limsup_{n\to\infty}\frac 1n \frac{H(g_1,\mcp_n)}{H(g_2,\mcp_n)}H(g_2,\mcp_n)\nonumber \\ &\leq & \sup_{x\in(0,\frac 12)}\frac{g_1(x)}{g_2(x)}\cdot\limsup_{n\to\infty}\frac 1n H(g_2,\mcp_n)=0.\nonumber \end{eqnarray}
Therefore we can assume that $\lim\limits_{n\to \infty} H(g_2,\mcp_n)=\infty$ 

Fix   $\varepsilon >0$. There exists $\delta>0$  such that for   $x\in (0,\delta]$ we have
\[\liminf_{x\to 0^+}\frac{g_1(x)}{g_2(x)}-\varepsilon<\frac{g_1(x)}{g_2(x)}\leq \limsup_{x\to 0^+}\frac{g_1(x)}{g_2(x)}+\varepsilon.\]

 Lemma \ref{lemmaalgebra} implies that
\begin{equation}\label{lilsg1g2} \liminf_{x\to 0^+}\frac{g_1(x)}{g_2(x)}-\varepsilon \leq \frac{\sum\limits_{A\in\mcp_n,\; \mu(A) < \delta}g_1(\mu(A))}{\sum\limits_{A\in\mcp_n,\;\mu(A)< \delta}g_2(\mu(A))} \leq \limsup_{x\to 0^+}\frac{g_1(x)}{g_2(x)} + \varepsilon.\end{equation}
Using (\ref{Gdelta}) for every $n>0$ we get
\[  \sum\limits_{ A\in\mcp_n,\; \mu(A)\geq \delta}g_i(\mu(B)) \leq \frac{1}{\delta}\overline{G_{\delta}^i}.\]
where $ \overline{G_{\delta}^i}:= \max\limits_{x\in [\delta,1]}g_i(x)$ for $i=1,2$. Therefore
\begin{eqnarray} \frac{\sum\limits_{ A\in\mcp_n:\;\mu(A) < \delta}g_1(\mu(A))}{\sum\limits_{ A\in\mcp_n:\;\mu(A) < \delta}g_2(\mu(A))+\frac{1}{\delta}\overline{G_{\delta}^2}}  &\leq &\nonumber  \frac{\sum\limits_{ A\in\mcp_n}g_1(\mu(A))}{\sum\limits_{ A\in\mcp_n}g_2(\mu(A))} \leq \frac{\sum\limits_{ A\in\mcp_n:\;\mu(A) < \delta}g_1(\mu(A))+\frac{1}{\delta}\overline{G_{\delta}^1}}{\sum\limits_{ A\in\mcp_n:\;\mu(A) < \delta}g_2(\mu(A))}. \end{eqnarray} and $\sum\limits_{ A\in\mcp_n:\;\mu(A) < \delta}g_2(\mu(A))\to \infty$ ($n\to\infty$). Dividing sums by $\sum\limits_{ A\in\mcp_n:\;\mu(A) < \delta}g_2(\mu(A))$ and from (\ref{lilsg1g2}) we obtain
\begin{eqnarray}\frac{\liminf\limits_{x\to 0^+}\frac{g_1(x)}{g_2(x)}-\varepsilon}{1+\overline{G_{\delta}^2}\left/\delta\sum\limits_{ A\in\mcp_n:\;\mu(A) < \delta}g_2(\mu(A))\right.} &\leq &\frac{\sum\limits_{ A\in\mcp_n}g_1(\mu(A))}{\sum\limits_{ A\in\mcp_n}g_2(\mu(A))} \nonumber\\ & \leq& \limsup\limits_{x\to 0^+}\frac{g_1(x)}{g_2(x)}-\varepsilon+\overline{G_{\delta}^1}\left/\delta\sum\limits_{ A\in\mcp_n:\;\mu(A) < \delta}g_2(\mu(A))\right.\nonumber\end{eqnarray}
Converging with $n$ to infinity we obtain:
\[\liminf_{x\to 0^+}\frac{g_1(x)}{g_2(x)}-\varepsilon\leq \liminf_{n\to\infty}\frac{H(g_1,\mcp_n)}{H(g_2,\mcp_n)}\leq \limsup_{n\to\infty}\frac{H(g_1,\mcp_n)}{H(g_2,\mcp_n)}\leq \limsup_{x\to 0^+}\frac{g_1(x)}{g_2(x)}+\varepsilon.\]
Therefore \begin{eqnarray}\left(\liminf_{x\to 0^+}\frac{g_1(x)}{g_2(x)}-\varepsilon\right)h(g_2,\mcp)&\leq & \liminf_{n\to\infty}\frac{H(g_1,\mcp_n)}{H(g_2,\mcp_n)} \cdot \limsup_{n\to\infty}\frac 1n H(g_2,\mcp_n)\nonumber \\ &\leq & \limsup_{n\to\infty}\frac 1n H(g_1,\mcp_n)  \nonumber \\
&\leq & \limsup_{n\to\infty}\frac{H(g_1,\mcp_n)}{H(g_2,\mcp_n)} \cdot \limsup_{n\to\infty}\frac 1n H(g_2,\mcp_n)\nonumber \\
&\leq &\left(\limsup_{x\to 0^+}\frac{g_1(x)}{g_2(x)}+\varepsilon\right)h(g_2,\mcp).\nonumber \end{eqnarray}
Therefore we obtain the assertion. In the case of infinite limit superior of the quotient $g_1(x)/g_2(x)$ we can repeat the above reasoning just omitting an upper bound for considered expressions.

If $\liminf\limits_{x\to 0^+}\frac{g_1(x)}{g_2(x)}>0$ and $h(g_2,\mcp)=\infty$, then $\varepsilon<\liminf\limits_{x\to 0^+}\frac{g_1(x)}{g_2(x)}$ and using similar arguments we obtain point 3.
\end{proof}

Using similar arguments we might obtain the answer in the case of infinite limit $\lim\limits_{x\to 0^+} g_1(x)/g_2(x)$ and positive dynamical $g_2$-entropy:
\begin{theo}\label{g1g2infty}
Let $g_1,g_2\in \mathcal{G}_0$ be such that $\lim\limits_{x\to 0^+}g_1(x)/g_2(x)=\infty$ and let a finite partition $\mcp$ has positive $g_2$-entropy. Then $h(g_1,\mcp)$ is infinite.
\end{theo}

Theorems \ref{hgp}, \ref{g1g2infty} imply few corollaries:
\begin{cor} \label{Cg1g2finite} If there exists the limit $\lim\limits_{x\to 0^+}\frac{g_1(x)}{g_2(x)}<\infty$, then $h(g_1,\mcp)=\lim\limits_{x\to 0^+}\frac{g_1(x)}{g_2(x)} \cdot h(g_2,\mcp)$.
\end{cor}

Let us define \[\mathcal{G}_{0}^{\infty}:=\left\{g\in \mathcal{G}_0\;\left|\; \lim_{x\to 0^+} \frac{g(x)}{\eta(x)}=\infty \right.\right\}.\]

If $g_1=g$, $g_2=\eta$, then the following we have the following corollary
\begin{cor} \label{Cgfinite}
 Let $\mcp\in\mathfrak{B}$ and $g\in\G$. Then
\begin{enumerate}
\item If $\Ci(g)<\infty$, then $h(g,\mcp)\geq \Ci(g)\cdot h(\mcp)$.
\item If $\Cs(g)<\infty$, then $h(g,\mcp)\in\left(\Ci(g)\cdot h(\mcp),\Cs(g)\cdot h(\mcp)\right)$. 
\item If  $g\in \mathcal{G}_0^{0}\cup \G^{\sh}$, then $h(g,\mcp)=\C(g) \cdot h(\mcp)$.
\item If $g\in\G^{\infty}$ and $h(\mcp)>0$, then $h(g,\mcp)=\infty$.
\end{enumerate}
\end{cor}

\begin{cor}
If $(X,\Sigma,\mu,T)$ has positive Kolmogorov-Sinai entropy and $g\in\G$ then:\\ $\Cs(g)<\infty$ $\Rightarrow$ $g$-entropy of any process $(X,\Sigma,\mu,T,\mcp)$ is finite $\Rightarrow$ $\Ci(g)<\infty$. 
\end{cor}

\begin{cor}
If $g\in\G^0\cup\G^{\sh}$, then $h(g,\mcp)=\lim\limits_{n\to\infty}\frac 1n H(g,\mcp_n)$.
\end{cor}

\subsection{Case of $g\in\G^{\infty}$}

We will prove that for every $g\in\G^{\infty}$, any aperiodic automorphism $T$ and every $\gamma \in\mathbb{R}$ there exists a partition $\mcp\in\mathfrak{B}$ such that $h(g,\mcp)\geq \gamma$. Since we omit the assumption of ergodicity we will use different techniques mainly based on the well-known Rokhlin Lemma which guarantees existence of so called Rokhlin towers of given height, covering sufficiently large part of $X$. Using such towers we will find lower estimations for $g$-entropy of a process similar to ones obtained  by Frank Blume in \cite{Blumenotpubl}, \cite{Blume97}, where he proposed, for a given sequence $(a_n)_{n=1}^{\infty}$ converging to infinity slower than $n$, a~construction of a~partition into two sets $\mcp$, for which $\lim\limits_{n\to\infty} H(\mcp_n)/a_n=\infty$. 

We will assume that we have an aperiodic system, i.e. system $(X,\Sigma,\mu,T)$ for which \[\mu\left(\{x\in X: \exists n\in\mathbb{N}\; T^nx=x\}\right)=0.\] 

If $M_0,\ldots,M_{n-1}\subset X$ are pairwise disjoint sets of equal measure, then $\tau=(M_0,M_1,\ldots,M_{n-1})$ is called  {\sf a tower}. 
If additionally $M_k=T^{-(n-k-1)} M_{n-1}$ for $k=1,\ldots,n-1$, then $\tau$ is called {\sf Rokhlin tower}.\footnote{It is also known as Rokhlin-Halmos or Rokhlin-Kakutani tower.}
By the same bold letter $\bm{\tau}$ we will denote the set $\bigcup_{k=0}^{n-1}M_k$. Obviously $\mu(\bm{\tau})=n\mu(M_{n-1})$. Integer $n$ is called {\sf the height } of tower $\tau$. Moreover for $i<j$ we define a subtower \[\tau_i^j:=\left(M_i,\ldots,M_j\right) \;\;\text{and}\;\;\bm{\tau}_i^j=\bigcup_{k=i}^j M_k. \] 
In aperiodic systems there exist Rokhlin towers of a given length and covering sufficiently large part of $X$:
\begin{lemma}[\cite{HeinemannSchmitt}] \label{Rohlin}
If $T$ is an aperiodic and surjective transformation of Lebesgue space $(X,\Sigma,\mu)$, then for every $\varepsilon>0$ and every integer $n\geq 2$ there exists a Rokhlin tower $\tau$ of height $n$ with $\mu(\bm{\tau})>1-\varepsilon$.
\end{lemma}


Our goal is to find a lower bound for the dynamical $g$-entropy of a given partition.
For this purpose we will use Rokhlin towers and we will calculate dynamical $g$-entropy with respect to a~given Rokhlin tower. This leads us to the following definition:
Let $\mcp$ be a finite partition of $X$ and  $F\in\Sigma$, then we define the (static) {\sf $g$-entropy of $\mcp$ restricted to $F$} as \[H_F (g,\mcp):=\sum_{B\in\mcp}g(\mu(B\cap F)).\]
\medbreak

The following lemma gives us estimation for $H(g,\mcp)$ from below by the value of $g$-entropy restricted to a subset of $X$. 

\begin{lemma} \label{lemmaH-Hzaw}
Let $g\in\mathcal{G}_0$. Let $\mcp$ be a finite partition such that there exists a set $E\in\mcp$ with $0<\mu(E)<1$. If $F\in\Sigma$, then
 \[H(g,\mcp)\geq H_{F}(g,\mcp)-\left|g_-'\left( 1/2\right)\right|-d_{\max},\]
where $d_{\max}:=\max\limits_{x,y\in[0,1]}|g(x)-g(y)|$.
\end{lemma}
\begin{proof}
By the mean value theorem we have  
\[g(\mu(A))-g(\mu(A\cap F))=g_-'(x_0^A)\left(\mu(A)-\mu(A\cap F)\right),\]
for any set of measure smaller or equal to $1/2$, where $x_0^A\in (\mu(A\cap F),\mu(A))$. Concavity of $g$ implies
\[\sum_{\mu(A)\leq 1/2}\left(g(\mu(A))-g(\mu(A\cap F))\right)\geq g_-'(1/2)\sum_{\mu(A)\leq 1/2}\mu(A\backslash F)\geq -|g_-'(1/2)|.\]
Eventually \[H(g,\mcp)-H_F(g,\mcp)+d_{\max} \geq -|g_-'(1/2)| \]
which completes the proof.
\end{proof}

The following lemma will play important rule in the proof of the main theorem of this section

\begin{lemma} \label{lemgPnEnFn}
Let $n\in\mathbb{N}$, $E\in\Sigma$. Suppose that $g\in\G$ is nonnegative in  $[0,\alpha]$, where $\alpha$ is some positive number. Then there exist  $\delta>0$ and $s\in(0,\alpha)$ such that
\[\left|H(g,\mcp_n^{E})-H(g,\mcp_n^{F})\right|\leq 1+\frac 2sd_{\max}\]
for every $F\in\Sigma$ s.t. $\mu(E\triangle F)<\delta$ ($\triangle$ denotes the symmetric difference), where $d_{\max}:=\max\limits_{x,y\in[0,1]}|g(x)-g(y)|$.
\end{lemma}
\begin{proof} 
Nonnegativity of $g$ for $x\in[0,\alpha]$ and its concavity imply that there exists   $s\in(0,\alpha)$ such that $g$ is nondecreasing in $[0,s]$. Fix $n\in\mathbb{N}$ and $E\in\Sigma$. 
There exists  $\delta\in (0,s)$ such that
\begin{equation}\label{gdelta2n}
g(\delta)<2^{-n}.
\end{equation}
Let $F\in\Sigma$ be such that $\mu(E\triangle F)\leq \delta$. Then for every $A\in\mcp_n^E$ i~$B\in\mcp_n^F$ we have
\begin{equation} \label{muABATBEF}
|\mu(A)-\mu(B)|\leq \mu(A\triangle B) \leq \mu(E\triangle F)\leq \delta.
\end{equation}
It is easy to see that for $x\in[0,s]$ the monotonicity and subadditivity of $g$ implies that
\begin{equation} \label{gyxgygx}
|g(y)-g(x)|\leq g(|y-x|).
\end{equation}
Define $\mathcal{D}_s=\{i\in \{1,\ldots,m\} \;|\; \mu(A_i)<s \text{ i } \mu(B_i)<s\}$.  From (\ref{gdelta2n}), (\ref{muABATBEF}), (\ref{gyxgygx}) and monotonicity of  $g$~in $[0,s]$ we obtain
\begin{eqnarray}
|H(g,\mcp_n^E)-H(g,\mcp_n^F)|&\leq & \sum_{i\in\mathcal{D}_s}|g(\mu(A_i))-g(\mu(B_i))|+\frac 2sd_{\max}\nonumber\\
&\leq & \sum_{i\in\mathcal{D}_s} g(|\mu(A_i)-\mu(B_i)|)+\frac 2sd_{\max} \nonumber\\ &\leq &  \sum_{i\in\mathcal{D}_s} g(\mu(A_i\triangle B_i))+\frac 2sd_{\max} \nonumber \\ &\leq &  \sum_{i=1}^m g(\mu(A_i\triangle B_i))+\frac 2sd_{\max} \nonumber \\ &\leq &  2^ng(\delta)+\frac 2sd_{\max} \leq 1+\frac 2sd_{max}\nonumber 
\end{eqnarray}
\end{proof}
To find the lower bound for the $g$-entropy of a partition we will construct so called {\bf independent sets}. We construct this set in the following way:
Let ${\bm \tau}$ be a tower of  height $m$. We divide the highest level of this tower ($M_{m-1}$) into two sets of equal measure let say  $I^{(m-1)}$ and $M_{m-1}\backslash I^{(m-1)}$. Next we consider $T^{-1}I^{(m-1)}$ and $T^{-1}(M_{m-1}\backslash I^{(m-1)})$. We divide each of them into two sets of equal measure and obtain sets $I_1^{(m-2)}, I_2^{(m-2)}, I_3^{(m-2)}, I_4^{(m-2)}$ and define set $I^{(m-2)}$ as the algebraic sum of two of those sets -- one taken from $T^{-1}I^{(m-1)}$ and the other  taken from $T^{-1}(M_{m-1}\backslash I^{(m-1)})$. We repeat this construction until we achieve the lowest level $M_0$.  (see Fig. \ref{Zbiorniezalezny}). We define set $I$ as $I:=\bigcup\limits_{j=0}^{m-1}I^{(j)}$. Such a set is called an independent set in  ${\bm \tau}$.
\begin{figure}[ht]
    \centering
  \includegraphics[width=0.53\textwidth, angle=0]{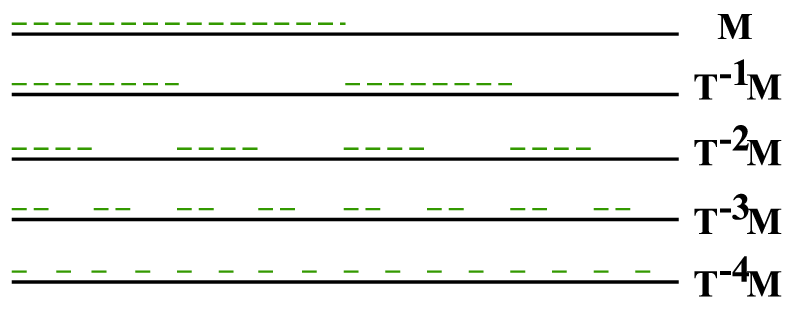}
\caption{ \label{Zbiorniezalezny}
Set $I$ (green dashes) in the tower  of height~5.}
\end{figure}

We can make this construction since every aperiodic system do not have atoms of positive measure and in every non-atomic Lebegue space for every measurable set $A$ and every $\alpha\in[0,\alpha]$ there exists $B\subset A$ such that $\mu(B)=\alpha$.

Now we give an estimation from below for the $g$-entropy restricted to a given Rokhlin tower. First, by $\mcp^I$ we will denote a partition into two sets $\{I,X\backslash I\}$, for a measurable set $I$. Then the following lemma is true.
\begin{lemma}\label{lemmaniezalwtau}
Let $\tau=\left(M,TM,\ldots,T^{2n-1}M\right)$ be Rokhlin tower of height $2n$, $I\in\Sigma$ be an independent set in $\tau$. If $g\in\G^{\infty}$ then
\[H_{\bm{\tau}_0^{n-1}}\left(g,\mcp_n^I\right)= \frac{\mu(\bm{\tau})}{2}\varphi\left(\frac{\mu(\bm{\tau})}{2^{n+1}}\right).\]
\end{lemma}

\begin{proof}
Independence of $I$ in $\tau$ implies that the partition \[\mcp_n^I\cap \bm{\tau}_0^{n-1}\] is a partition of  $\bm{\tau}_0^{n-1}$ into $2^n$ sets of equal measure $2^{-n}\mu(\bm{\tau}_0^{n-1})$. Therefore 
\begin{eqnarray}
H_{\bm{\tau}_0^{n-1}}\left(g,\mcp_n^I\right) &=& \sum_{A\in\mcp_n^I} g\left(\mu(A\cap\bm{\tau}_0^{n-1})\right)\nonumber \\ &=& 2^n g\left(\frac{\mu(\bm{\tau}_0^{n-1})}{2^n}\right) = \mu\left(\bm{\tau}_0^{n-1}\right)\varphi\left(\frac{\mu(\bm{\tau}_0^{n-1})}{2^n}\right) \nonumber \\ &=&\frac{\mu(\bm{\tau})}{2}\varphi\left(\frac{\mu(\bm{\tau})}{2^{n+1}}\right). \nonumber
\end{eqnarray}

\end{proof}

\begin{theo}\label{hgpgeqM}
Let $g\in\G^{\infty}$ and $T$ be an aperiodic, surjective automorphism of a Lebesgue space $(X,\Sigma,\mu)$ and let $\gamma\in\mathbb{R}$. Then there exists a partition $\mcp\in\mathfrak{B}$ such that \[h(g,\mcp)\geq \gamma.\] 
\end{theo}

\begin{proof}
We will prove that
for any $\gamma>0$ there exists a partition $\mcp^E=\{E,X\backslash E\}$ such that $h(g,\mcp)\geq \gamma$. We define recursively a sequence of sets $E_n\in \Sigma $. Let  \[E_0:=\emptyset,\;\; N_0:=\delta_0:=1.\] Let $n>0$ and assume that we have already defined $E_{n-1}$, $N_{n-1}$ and $\delta_{n-1}$. Using Lemma \ref{lemgPnEnFn} we can choose  $\delta_n>0$ such that
\begin{equation}\label{deltan}
\delta_n<\frac 12 \delta_{n-1}
\end{equation}
\begin{equation} \label{HNnEnF}
\left|H\left(g,\mcp_{N_n}^{E_{n-1}}\right)-H\left(g,\mcp_{N_n}^{F}\right)\right| <1+\frac{2}{s}d_{\max}.
\end{equation}
for any $F\in\Sigma$, for which $\mu(E_{n-1}\triangle F)<2\delta_n$. 

Since \[\lim_{x\to 0^+}\frac{g(x)}{\eta(x)}=\infty,\] we can choose such  $N_n\in\mathbb{N}$ that 
\begin{equation} \label{eqnvarphi}
\frac{\varphi\left(\delta_n 2^{-N_n-1}\right)}{\varphi_{\eta}\left(\delta_n 2^{-N_n-1}\right)}>\frac{2\gamma}{\delta_n \log 2}.
\end{equation}

By Lemma \ref{Rohlin} there exists  $M_n\in \Sigma$, such that $\tau_n=\left(M_n,TM_n,\ldots,T^{2N_n-1}M_n\right)$ is a Rokhlin tower of measure $\mu(\bm{\tau}_n)=\delta_n$. Let  $I_n\subset \bm{\tau}_n$ be an independent set in $\tau_n$ and \[E_n:=\left(E_{n-1}\backslash \bm{\tau}_n\right)\cup I_n.\] Then   \[\mu(E_{n-1}\triangle E_n)\leq \mu(\bm{\tau}_n)=\delta_n.\] for all positive integers $n$. By (\ref{deltan}) we have $\delta_n<2^{-n}$ and we conclude that $\left(\bm{1}_{E_n}\right)_{n=0}^{\infty}$ is a Cauchy sequence in $L_1(X)$. Therefore there exist $E\in\Sigma$ such that $\bm{1}_{E_n}$ converges to $\bm{1}_E$. For this set we have \[\mu\left(E_n\triangle E\right)\leq \sum_{k=n+1}^{\infty}\mu\left(E_k\triangle E_{k-1}\right)\leq \sum_{k=n+1}^{\infty}\delta_k<2\delta_{n+1}.\]
Since $E_n\cap\tau_n =I_n$, applying (\ref{HNnEnF}) and Lemmas \ref{lemmaH-Hzaw}, \ref{lemmaniezalwtau} we obtain that for $N_n$ such that $\delta_n\cdot2^{-N_n-1}<s$:
\begin{eqnarray}
H(g,\mcp_{N_n}^E) &\geq & H(g,\mcp_{N_n}^{E_n})-1-\frac 2sd_{\max}\nonumber \\ &\geq &H_{\left({\bm\tau}_n\right)_{0}^{N_n-1}}\left(g,\mcp_{N_n}^{E_n}\right)-\left|g_-'\left( 1/2\right)\right|-\left(\frac 2s+1\right)d_{\max}-1\nonumber \\ &\geq &H_{\left({\bm\tau}_n\right)_{0}^{N_n-1}}\left(g,\mcp_{N_n}^{I_n}\right)-\left|g_-'\left( 1/2\right)\right|-\left(\frac 2s+1\right)d_{\max}-1\nonumber\\ &\geq & \left[\frac{\mu(\bm{\tau}_n)\ln 2}{2}\left(N_n+1\right)-\frac{\mu(\bm{\tau}_n)\ln\mu(\bm{\tau}_n)}{2}\right]\cdot\frac{\varphi\left(\mu(\bm{\tau}_n)2^{-N_n-1}\right)}{-\ln\left(\mu(\bm{\tau}_n)2^{-N_n-1}\right)}\nonumber \\
&& -\left|g_-'\left( 1/2\right)\right|-\left(\frac 2s+1\right)d_{\max}-1 \nonumber\\
&\geq & \frac{\ln 2}{2}\cdot \delta_n\cdot \left(N_n+1\right)\cdot\frac{\varphi\left(\delta_n 2^{-N_n-1}\right)}{\varphi_{\eta}\left(\delta_n 2^{-N_n-1}\right)}-\left|g_-'\left(1/2\right)\right|\nonumber \\ && -\left(\frac 2s+1\right)d_{\max}-1. \nonumber
\end{eqnarray}
From (\ref{eqnvarphi}) we obtain that
\begin{eqnarray}
\lim_{n\to\infty}\frac{H(g,\mcp_{N_n}^E)}{N_n}&\geq &  \frac{\ln 2}{2} \lim_{n\to\infty} \delta_n\cdot \frac{N_n+1}{N_n}\cdot\frac{\varphi\left(\delta_n2^{-N_n-1}\right)}{\varphi_{\eta}\left(\delta_n2^{-N_n-1}\right)} \geq  \gamma. \nonumber
\end{eqnarray}
\end{proof}

\subsection{Bernoulli shifts.}
Let $\mathcal{A}=\{1,\ldots,k\}$ be a finite alphabet. Let $X=\{x=\{x_i\}_{i=-\infty}^{\infty}\colon x_i\in\mathcal{A}\}$ and  $\sigma$ be a left shift \[\sigma(x)_i=x_{i+1}.\] For any $s\leq t$ and block $[\omega_0,\ldots,\omega_{t-s}]$ with $a_i\in\mathcal{A}$ we define a cylinder \[C_s^t(\omega_0,\ldots,\omega_{t-s})=\{x\in X: x_i=\omega_{i-s}\;\text{for}\; i=s,\ldots,t\}.\] We consider the Borel $\sigma$-algebra with respect to the metric, which is given by $d(x,y)=2^{-N}$, where $N=\min\{|i|:x_i\neq y_i\}$. One can show that Borel $\sigma$-algebra is the minimal $\sigma$-algebra containing all cylindrical sets. Let $p=\left(p_1,\ldots,p_k\right)$ be a probability vector, i.e. $p_i\geq 0$ for any $i$ and $\Sigma p_i=1$. We define a measure $\rho=\rho (p)$ on $\mathcal{A}$ by setting $\rho(\{i\})=p_i$. Then $\mu_p$ is a corresponding product measure on $X=\mathcal{A}^{\mathbb{Z}}$.
Thus, the static  $g$-entropy of a partition  $\mcp^{\mathcal{A}}=\{[1],[2],\ldots,[k]\}$ is equal to
\[H_{\mu_p}\left(g,\mcp_n^{\mathcal{A}}\right)=\sum_{\omega\in\mathcal{A}^n}g\left(\mu(C_0^{n-1}(\omega_0,\ldots,\omega_{n-1}))\right)=\sum_{\omega\in\mathcal{A}^n} g\left(p_{\omega_0}\cdots p_{\omega_{n-1}}\right),\]
where $\omega=(\omega_0,\ldots,\omega_{n-1})$. By the concavity of the function $g$ we have
\[H_{\mu_p}(g,\mcp_n^{\mathcal{A}})\leq \varphi\left(\frac{1}{k^n}\right)\]
where  equality holds only when  $p=p^*=\left(\frac 1k,\ldots, \frac 1k\right)$. 
Before calculating the dynamical $g$-entropy of the partition $\mcp^{\mathcal{A}}$ with respect to measure $\mu_{p^*}$, we give the following lemma, which proof will be given later:
\begin{lemma} \label{limsupMn}
If $g\in\G$, then \[\Cs(g)=\limsup_{n\to\infty}\frac{g(\kappa^{-n})}{\eta(\kappa^{-n})}\;\;\text{and}\;\; \Ci(g)=\liminf_{n\to\infty}\frac{g(\kappa^{-n})}{\eta(\kappa^{-n})}\] for any   $\kappa>1$.
\end{lemma}

Therefore, applying Lemma \ref{limsupMn} for the partition $\mcp^{\mathcal{A}}$ and $\kappa=k$ we obtain
\begin{equation} \label{entrBernoulli} h_{\mu_{p^*}}\left(g,\mcp^{\mathcal{A}}\right)=\limsup_{n\to\infty}\frac 1n \varphi\left(\frac{1}{k^n}\right)=\left\{\begin{array}{ll} \Cs(g) \cdot \log k, & \;\; \text{if}\;\; \Cs(g)<\infty;\\ \infty, &\;\;\text{otherwise.}\end{array}\right.\end{equation}

\begin{remark} 
If we consider lower limit instead of the upper limit we would obtain
\[ \liminf_{n\to\infty}\frac 1n \varphi\left(\frac{1}{k^n}\right)=\left\{\begin{array}{ll} \Ci(g)\cdot\log k, & \;\; \text{if}\;\; \Ci(g)<\infty;\\ \infty, &\;\;\text{otherwise.}\end{array}\right.\]
Therefore we can not replace an upper limit by the limit in the definition of the dynamical $g$-entropy.
\end{remark}

\begin{proof}[Proof of Lemma \ref{limsupMn}] We will show the equality for the upper limit. Proof of the equality for the lower limit is similar.
Let $(x_n)_{n=1}^{\infty}$ and $(m_n)_{n=1}^{\infty}$ be such that $\limsup\limits_{n\to\infty}g(x_n)/\eta(x_n)=c$ and $x_n\in\left(\kappa^{-m_n},\kappa^{-m_n+1}\right)$ for every $n\in\mathbb{N}$. Then $-\log x_n \geq -\log \kappa^{-m_n+1}$.  Every function  $g\in \G$ is quasihomogenic, so for every positive integer $n$ occurs \[ \label{quasigxn}\frac{g(x_n)}{x_n}<\frac{g(\kappa^{-m_n})}{\kappa^{-m_n}}.\]

Therefore
\begin{eqnarray}
\frac{g(x_n)}{\eta(x_n)}&=&\frac{g(x_n)}{x_n}\frac{1}{-\log x_n}\leq \frac{g(\kappa^{-m_n})}{\kappa^{-m_n}}\frac{1}{(m_n-1)\log \kappa} \nonumber \\ &=& \frac{g(\kappa^{-m_n})}{\eta\left(\kappa^{-m_n}\right)}\cdot\frac{m_n}{m_n-1},\nonumber
\end{eqnarray}
and
 \[\limsup_{x\to 0^+}\frac{g(x)}{\eta(x)}=\limsup_{n\to\infty}\frac{g(\kappa^{-n})}{\eta(\kappa^{-n})}.\]
\end{proof}

\section{Kolmogorov-Sinai entropy like invariant}

The basic tool in the ergodic theory is Kolmogorov-Sinai entropy defined as a~supremum of Shannon dynamical entropies over all finite partitions: \[h_{\mu}(T)=\sup_{\mcp\;-\;\text{finite}}h(T,\mcp).\]
It is invariant under metric isomorphism. Following the  Kolmogorov proposition we take the supremum over all partitions of dynamical $g$-entropy of a partition. For a given system $(X,\Sigma,\mu,T)$ we define
\begin{equation}\label{hgKS}
h_{\mu}(g,T)=\sup_{\mcp\;-\;\text{finite}}h(g,T,\mcp)
\end{equation}
and call it {\sf the measure-theoretic $g$-entropy  of transformation $T$ with respect to measure $\mu$}.

It is easy to see that it is an isomorphism invariant.  Ornstein and Weiss \cite{OrnsteinWeiss} showed the striking result that measure-theoretic entropy is the only finitely observable invariant for the class of all ergodic processes. More precisely -- every finitely observable invariant for a class of all ergodic processes is a continuous function of entropy. Of course in the case of $g\in\G^0\cup\G^{\sh}$ by Corollary \ref{Cgfinite} we have \[h_{\mu}(g,T)=\lim_{x\to 0^+}\frac{g(x)}{\eta(x)}\cdot h_{\mu}(T).\]
  We will show that for a wider class of functions, namely for functions for which \[\Cs(g)=\limsup_{x\to 0^+}\frac{g(x)}{\eta(x)}<\infty\] we have \[h_{\mu}(g,T)=\Cs(g)\cdot  h_{\mu}(T)\]
for any ergodic transformation $T$. This shows that the measure-theoretic $g$-entropy is in fact finitely observable: one might simply compose the entropy estimators \cite{Weiss} with the linear function itself. Our proof will be similar to the proof of \cite[Thm 1.1]{TakensVerb1998} where Takens and Verbitski showed that for ergodic transformations supremum over all finite partitions of dynamical R\'{e}nyi entropies of order $\alpha>1$ are equal to the measure-theoretic entropy of $T$ with respect to measure $\mu$.

Let us introduce necessary definitions.
Let  $T_i$ be automorphisms of Lebesgue space $(X_i,\Sigma_i,\mu_i)$ for $i=1,2$ respectively. Then we say that  $T_2$ is a  {\sf factor} of transformation $T_1$, if there exists a homomorphism $\phi\colon X_1\mapsto X_2$ such that \[\phi T_1=T_2\phi \;\; \mu_1 \;\text{a.e. on}\;\; X_1.\]
Suppose that $T_2$ is a factor of $T_1$ under homomorphism $\phi$. Then for an arbitrary finite partition $\mcp$ of $X_2$ we have 
\[H\left(g,\bigvee_{i=0}^{k-1}T_2^{-i}\mcp\right)=H\left(g,\bigvee_{i=0}^{k-1}\phi^{-1}T_2^{-i}\mcp\right)=H\left(g,\bigvee_{i=0}^{k-1}T_1^{-i}\phi^{-1}\mcp\right).\]
Hence $h(g,T_2,\mcp)=h(g,T_1,\phi^{-1}\mcp)$. Therefore \[h_{\mu}(g,T_2)=\sup_{\mcp - \text{finite}}h(g,T_2,\mcp)=\sup_{\mcp - \text{finite}} h(g,T_1,\phi^{-1}\mcp)\leq h(g,T_1).\]
This implies the following proposition:
\begin{prop} \label{factorhmu}
If $T_2$ is a factor of $T_1$, then for every function $g\in\G$  \[h_{\mu}(g,T_2)\leq h_{\mu}(g,T_1).\]
\end{prop}

\subsection{Measure-theoretic $g$-entropies for Bernoulli automorphisms.}
An automorphism $T$ on $(X,\Sigma,\mu)$ is called {\sf Bernoulli automorphism} if it is isomorphic to some Bernoulli shift. The crucial role in the proof of the main theorem of this section (Theorem \ref{twhKSg}) will play a well-known theorem due to Sinai:
\begin{theo}[Sinai,  \cite{Sinaj}] \label{Sinaj}
Let  $T$ be an arbitrary ergodic automorphism of some Lebesgue space $(X,\Sigma,\mu)$. Then each Bernoulli automorphism  with $h_{\mu}(T_1)\leq h_{\mu}(T)$ is a factor of the automorphism  $T$.
\end{theo}

The following proposition will play a crucial role in our considerations:

\begin{prop} \label{erghmulogm}
Let $T$ be an arbitrary ergodic automorphism with $h_{\mu}(T)\geq \log M$ for some integer $M\geq 2$. Then for every $g\in\G$  \[h_{\mu}(g,T)\geq \Cs(g)\cdot\log M.\] 
\end{prop}


\begin{proof}
Consider a shift $\sigma$ over all infinite sequences from the alphabet $\mathcal{A}=\{0,1\ldots,M-1\}$ with the corresponding Bernoulli measure generated by $p_1=\ldots = p_M=\frac 1M$. It is easy to see that $h_{\mu}(\sigma)=\log M$. From Theorem  \ref{Sinaj} we conclude that $\sigma$ is a factor of $T$. Therefore applying formula  (\ref{entrBernoulli}) we obtain \[h_{\mu}(g,T)\geq h_{\mu}(g,\sigma)\geq h(g,\sigma,\mcp^{\mathcal{A}})=\limsup_{n\to\infty}\frac 1n \varphi\left(M^{-n}\right)=\log M \cdot \limsup_{n\to\infty}\frac{\varphi\left(M^{-n}\right)}{\varphi_{\eta}\left(M^{-n}\right)}.\]
Applying Lemma \ref{limsupMn} completes the proof.
\end{proof}

\subsection{Main theorem}
Our goal in this section is the following result:
\begin{theo}\label{twhKSg}
Let $T$ be an ergodic automorphism of Lebesgue space $(X,\Sigma,\mu)$, and   $g\in\G$ be such  that $ \Cs(g) \in(0,\infty)$ Then \[h_{\mu}(g,T)=\left\{\begin{array}{ll}\Cs(g)\cdot h_{\mu}(T),& \; \text{if}\; h_{\mu}(T)<\infty,\\ \infty,& \; \text{otherwise}.\end{array}\right.\]
 If $g\in\G^0$, then $h_{\mu}(g,T)=0$. If $g\in\G$ is such that $\Cs(g)=\infty$ and  $T$ has positive measure-theoretic entropy, then $h_{\mu}(g,T)=\infty$.
\end{theo}
Moreover for $g\in\G^{\infty}$ from Theorem \ref{hgpgeqM} we have
\begin{cor}
Let $g\in\G^{\infty}$. If $(X,T)$ is aperiodic and surjective than $h_{\mu}(g,T)=\infty$.
\end{cor}
To prove Theorem \ref{twhKSg} we need first few preliminary lemmas.
\begin{lemma} \label{lemma73}
If $T$ is an automorphism of the Lebesgue space $(X,\Sigma,\mu)$, then for every $g\in\G$   \[h_{\mu}(g,T^m)\leq mh_{\mu}(g,T).\]
\end{lemma}
\begin{proof}
Let  $\mcp$ be a finite partition, $m\in\mathbb{N}$. We have
\begin{eqnarray}
h(g,T,\mcp)&=&\limsup_{k\to\infty}\frac 1k H(g,\mcp\vee T^{-m}\mcp\vee\ldots\vee T^{-m(k-1)}\mcp) \nonumber \\ &=& \lim_{n\to\infty}\sup_{k\geq n}\frac 1k H(g,\mcp\vee T^{-m}\mcp\vee\ldots\vee T^{-m(k-1)}\mcp). \nonumber
\end{eqnarray}
Fix $k\in \mathbb{N}$. Then the partition $\bigvee\limits_{i=0}^{n-1}T^{-i}\mcp$ is a refinement of  $\mcp\vee T^{-m}\mcp\vee\ldots\vee T^{-m(k-1)}\mcp$ for $n=km,\ldots,km+m-1$. Therefore
\begin{eqnarray}
\frac 1k H(g,\mcp\vee T^{-m}\mcp\vee\ldots\vee T^{-m(k-1)}\mcp) &\leq & \frac 1k H(g,\mcp_{n-1})=\frac nk\frac 1n H(g,\mcp_{n-1}) \nonumber \\
&\leq & \frac{km+m-1}{k}\frac 1n H(g,\mcp_{n-1})\nonumber \\ &\leq & m\left(1+\frac 1k\right)\frac 1n H(g,\mcp_{n-1})\label{72}
\end{eqnarray}
for $n=km,\ldots,km+m-1$. Let introduce the following notation:
\[c_k:=\frac 1k H(g,\mcp\vee T^{-m}\mcp\vee\ldots\vee T^{-m(k-1)}\mcp),\;\;\; a_n :=\frac 1n H(g,\mcp_{n-1}). \]
Then we can rewrite (\ref{72}) in the form \begin{equation} \label{73ref}c_k\leq m\left(1+\frac 1k\right)a_n\end{equation}
for $n=km,\ldots,km+m-1$. Taking supremum in (\ref{73ref}) we obtain
\[\sup_{l\geq k}c_l \leq m\left(1+\frac 1k\right) \sup_{n=lm,\ldots,lm+m-1}a_n\leq m\left(1+\frac 1k\right) \sup_{n\geq km}a_n.\]
Therefore \[\liminf_{k\to\infty}c_k\leq m\limsup_{n\to\infty}a_n,\]
and this is equivalent to the statement
\[h(g,T^m,\mcp)\leq m h(g,T,\mcp).\]
Taking supremum over all finite partitions we obtain the assertion.
\end{proof}

Next lemma will be just a weaker version of Theorem \ref{twhKSg}. 
\begin{lemma} \label{calkergcg}
If an automorphism $T^m$ of a Lebesgue space $(X,\Sigma,\mu)$ is ergodic for every $m\in\mathbb{N}$, then for every function $g\in\G$, such that $\Cs(g)<\infty$ holds  \[h_{\mu}(g,T)=\Cs(g)\cdot h_{\mu}(T).\]
If $g\in\G^0$, then $h_{\mu}(g,T)=0$. 
If $g\in\G$ is such that $\Cs(g)=\infty$ and $T$ has positive Kolmogorov-Sinai entropy, then $h_{\mu}(g,T)=\infty$.
\end{lemma}

\begin{proof}
Case of  $g\in\G^0$ follows from Corollary \ref{Cgfinite}.
Suppose that there exists such $g\in\G\backslash \G^0$ which fullfills assumptions of lemma and for which we have \[\Cs(g)\cdot h_{\mu}(T)-h_{\mu}(g,T)>0.\]
Then applying Lemma \ref{lemma73} to the transformation $T^m$ and using equality \mbox{$h_{\mu}(T^m)=mh_{\mu}(T)$} (see \cite[Thm 4.3.16]{KH}) we obtain
\[\Cs(g)h_{\mu}(T^m)-h_{\mu}(g,T^m)\geq m\left(\Cs(g)h_{\mu}(T)-h_{\mu}(g,T)\right) \rightarrow \infty\;\; \text{as}\;\;m\to\infty.\]
Therefore for sufficiently large $m$ there exists an integer $M$ for which 
\begin{equation}\label{calerg1} h_{\mu}(g,T^m)\leq m h_{\mu}(g,T)<\Cs(g)\log M\leq m \Cs(g)h_{\mu}(T)=\Cs(g)h_{\mu}(T^m).\end{equation}
Proposition \ref{erghmulogm} applied to the transformation $T^m$ guarantees that for every  $g\in \G$ with positive (finite) $\Cs(g)$ we have \begin{equation}\label{calerg2} h_{\mu}(g,T^m)\geq \Cs(g)\log M.\end{equation}
Comparing (\ref{calerg1}) and (\ref{calerg2}) we obtain the contradiction, which implies that \[h_{\mu}(g,T)=\Cs(g)h_{\mu}(T).\]
If  $\Cs(g)=\infty$ and $h_{\mu}(T)>0$ then  there exists such integer $m>0$ that \[h_{\mu}(T^m)=mh_{\mu}(T)>\log M\]  and by Proposition \ref{erghmulogm} and Lemma \ref{lemma73} \[h_{\mu}(g,T)=h_{\mu}(g,T^m)=\infty\] which completes the proof.
\end{proof}

\begin{proof}[Proof of Theorem \ref{twhKSg}]
If $h_{\mu}(T)=0$ theorem is true, since for any partition $\mcp$ we have \[0\leq h(g,\mcp)\leq \Cs(g) h(\mcp)=0.\]
Suppose that $0<h_{\mu}(T)<\infty$. Automorphism $T$ is ergodic. Therefore it has factor which is a Bernoulli automorphism  $T'$ with entropy $h_{\mu}(T)=h_{\mu}(T')$. Every Bernoulli automorphism is mixing, so $T^m$ is ergodic for each $m$.
Applying Lemma \ref{calkergcg} we obtain \[h_{\mu}(g,T')=\Cs(g) h_{\mu}(T')=\Cs(g) h_{\mu}(T).\]
Since $T'$ is a factor of $T$, so  Proposition \ref{factorhmu} implies that
\[\Cs(g)h_{\mu}(T)=\Cs(g)h_{\mu}(T')=h_{\mu}(g,T')\leq  h_{\mu}(g,T)\leq \Cs(g)h_{\mu}(T)\]
which completes the proof of the case of finite  $h_{\mu}(T)$. If $h_{\mu}(T)=\infty$, then Proposition \ref{erghmulogm} implies that \[h_{\mu}(g,T)\geq \Cs(g)\log M\]
for every $M>0$ and the theorem is proved.
\end{proof}

\subsection{Generator theorem counterpart} In the case of $g\in\G^{\infty}$ there is no counterpart of a Kolmogorov-Sinai generator theorem, which says that the measure-theoretic entropy of the transformation $T$ is realised on every generator of the $\sigma$-algebra $\Sigma$. 
Let us consider Sturm shifts -- shifts which model translations of the circle  $\mathbb{T}=[0,1)$. Let $\beta \in [0,1)$ and consider the translation $\phi_{\beta}\colon [0,1)\mapsto [0,1)$ defined by $\phi_{\beta}(x)=x+\beta \;(\mod 1)$. Let $\mcp$ denote the partition of $[0,1)$ given by $\mcp=\{[0,\beta),[\beta,1)\}$. Then we associate a binary sequence to each $t\in[0,1)$ according to its itinerary relative to $\mcp$; that is we associate to $t\in[0,1)$ the bi-infinite sequence $x$ defined by $x_i=0$ if $\phi_{\beta}^i(t)\in[0,\beta)$ and $x_i=1$ if $\phi_{\beta}^i(t)\in[\beta,1)$. The set of such sequences is not necessary closed, but it is shift-invariant and so its closure is a~shift space called Sturmian shift. If $\beta$ is irrational, then Sturmian shift is minimal, i.e. there is no proper subshift. Moreover for a minimal Sturmian shift, the number of $n$-blocks which occur in an infinite shift space is exactly $n+1$. Therefore  for zero-coordinate partition $\mcp^{\mathcal{A}}$, which is a finite generator of $\sigma$-algebra $\Sigma$ and for any function $g\in\G$ we have
\[H(g,\mcp_n^{\mathcal{A}})=\sum_{A\in\mcp_n^{\mathcal{A}}}g(\mu_S(A))\leq  \varphi\left(\frac{1}{n+1}\right)\]  where $\mu_S$ is the unique invariant measure for Sturm shift.  Thus,
\[h(g,\mcp^{\mathcal{A}})\leq \limsup_{n\to\infty}\frac{n+1}{n}g\left(\frac{1}{n+1}\right)=0.\] 
On the other hand since it is strictly ergodic (and thus aperiodic)  Theorem \ref{hgpgeqM} implies that  for any function $g\in\G^{\infty}$ \[h_{\mu}(g,T)=\infty,\] therefore we have a finite generator, for which the supremum is not attained.


\end{document}